\documentclass[11pt]{amsart}
\usepackage{amssymb}
\usepackage{mathtools}
\usepackage{enumerate}
\numberwithin{equation}{section}
\DeclarePairedDelimiter\abs{\lvert}{\rvert}%
\DeclarePairedDelimiter\norm{\lVert}{\rVert}%
\DeclarePairedDelimiter\babs{\biggl\lvert}{\biggr\rvert}%
\DeclarePairedDelimiter\bnorm{\biggl\lVert}{\biggr\rVert}%
\DeclarePairedDelimiter\bbnorm{\bigl\lVert}{\bigr\rVert}%

\makeatletter
\let\oldabs\abs
\def\abs{\@ifstar{\oldabs}{\oldabs*}}
\let\oldnorm\norm
\def\norm{\@ifstar{\oldnorm}{\oldnorm*}}
\makeatother

\theoremstyle{plain}
\newtheorem{thm}{Theorem}[section]

\newtheorem{lem}[thm]{Lemma}
\newtheorem{prop}[thm]{Proposition}
\newtheorem{cor}[thm]{Corollary}
\theoremstyle{definition}

\theoremstyle{remark}

\newtheorem*{remark}{Remark}
\newtheorem*{notation}{Notation}

\newcommand\R{\mathbb{R}}
\newcommand\Rp{\R_+}

\newcommand\Z{\mathbb{Z}}
\newcommand\N{\mathbb{N}}

\newcommand\G{\mathcal{G}}
\newcommand\con{*_d}

\newcommand{\rad}{ _{\operatorname{rad}}}
\newcommand{\s}{\mathcal{S}}
\newcommand{\ben}{\begin{enumerate}[(i)]}
\newcommand{\een}{\end{enumerate}}
\newcommand{\ft}{\mathcal{F}}
\newcommand{\fto}{\mathcal{F}_\mathbb{R}}
\newcommand{\ftd}{\mathcal{F}_{\mathbb{R}^d}}
\newcommand{\ftn}{\mathcal{F}_{\mathbb{R}^n}}
\newcommand{\ift}{\mathcal{F}^{-1}_\mathbb{R}}

\newcommand{\fd}{D^{\vec{\alpha}}}

\newcommand{\supp}{\operatorname{supp}}

\newcommand{\inn}[1]{\langle #1 \rangle}
\newcommand{\h}{\mathcal{H}_{\ve{d}}}
\newcommand{\hh}{\mathcal{H}}
\newcommand{\han}{{\mathcal{H}_d}}

\newcommand{\lp}{L^p(\mu_d,H)}
\newcommand{\lpl}{L^{p,2}(\mu_d,H)}
\newcommand{\lpt}{{L^{p}(\mu_d,L^2_t(H))}}

\newcommand{\pare}[2]{\biggl( #1 \biggr) ^{1/#2} } 
\newcommand{\ve}[1]{\vec{#1}\,}
\newcommand{\les}{\lesssim}

\begin{document}
\title[Square functions associated with Hankel multipliers]{Endpoint bounds of square functions associated with Hankel multipliers}
\author{Jongchon Kim}
\thanks{Supported in part by NSF grant 1201314} 
\address{Department of Mathematics, University of Wisconsin-Madison, Madison,
WI,
53706 USA}
\email{jkim@math.wisc.edu}
\begin{abstract} We prove endpoint bounds for the square function associated with radial Fourier multipliers acting on $L^{p}$ radial functions. This is a consequence of endpoint bounds for a corresponding square function for Hankel multipliers. We obtain a sharp Marcinkiewicz-type multiplier theorem for multivariate Hankel multipliers and $L^p$ bounds of maximal operators generated by Hankel multipliers as corollaries. The proof is built on techniques developed by Garrig\'{o}s and Seeger for characterizations of Hankel multipliers.
\end{abstract}
\maketitle

\section{Introduction} \label{sec:intro}
Let $S^\lambda_t$ be the Bochner-Riesz mean of order $\lambda>0$ defined by
\begin{equation*}
\ft[S^{\lambda}_t f] (\xi) = \biggl(1-\frac{|\xi|^2}{t^2}\biggr)^\lambda_+ \ft f(\xi)
\end{equation*}
for $t>0$, where $\ft$ denotes the Fourier transform $\ft f(\xi) = \int_{\R^d} f(x)e^{ix\cdot \xi} dx$. One is interested in the convergence $S^\lambda_t f \to f$ as $t \to \infty$ in various sense. In this regard, $L^p$ estimates of $S^\lambda := S^\lambda_1$ and the maximal operator $S_*^\lambda f(x) = \sup_{t>0} |S_t^\lambda f(x)|$ have been studied extensively. For $\lambda$ below the critical index $\frac{d-1}{2}$, it is conjectured that $S^\lambda$ is bounded on $L^p$ if and only if $\frac{2d}{d+1+2\lambda}<p<\frac{2d}{d-1-2\lambda}$ and
that $S^\lambda_*$ is bounded for the same $p$-range for $p\geq 2$. 

Although the conjectures remain open in the full $p$-range for $d\geq 3$, they are indeed theorems for $d=2$ by Carleson and Sj\"{o}rin \cite{CaSj} and Carbery \cite{CarM}, respectively. In addition, the work by Carbery, Gasper, and Trebels \cite{CGT} and Carbery \cite{CarR} shows that the results for $d=2$ are consequences of more general multiplier theorems which apply to all radial Fourier multipliers. This involves the square function $G^\alpha$,
\begin{equation*}
G^\alpha f(x) = \biggl(\int_0^\infty |R^\alpha_t f(x)|^2 \frac{dt}{t}\biggr)^{1/2},
\end{equation*}
where $\ft[R^\alpha_t f](\xi) = \frac{|\xi|}{t}(1-\frac{|\xi|}{t})_+^{\alpha-1} \ft f(\xi)$. The square function (with $S^\alpha_t f - S^{\alpha-1}_t f$ in place of $R^\alpha_t f$) was introduced by Stein \cite{St} for the $L^2$-bound of $S^\lambda_*$ for $\lambda>0$.

Let $m$ be a bounded function on $\R_+:=(0,\infty)$ and $\mathcal{T}_m$ be the operator defined by,
\begin{equation*}
\ft [\mathcal{T}_m f] (\xi) = m(|\xi|) \ft f(\xi).
\end{equation*}
Then for $\alpha>1/2$ and a fixed non-trivial smooth function $\phi$ supported on $[1,2]$, there is a pointwise estimate 
\begin{equation}\label{eqn:pom}
g[\mathcal{T}_m f](x) \leq C \sup_{t>0} \norm{m(t\cdot) \phi}_{L^2_\alpha(\R)} G^\alpha f(x),
\end{equation}
where $g$ is a standard Littlewood-Paley square function and $L^2_\alpha(\R)$ is the $L^2$-Sobolev space. Since $\norm{g(\mathcal{T}_mf)}_{L^p(\R^d)}$ is comparable to $\norm{\mathcal{T}_mf}_{L^p(\R^d)}$ for $1<p<\infty$, \eqref{eqn:pom} reduces $L^p$ estimates of $\mathcal{T}_m$ to the study of $G^\alpha$, which is independent of a specific multiplier $m$. Moreover, it was shown by \cite{CarR} that $G^\alpha$ controls maximal operator generated by $\mathcal{T}_m$ by a pointwise estimate, which gives effective $L^p$ bounds for maximal functions $S^\lambda_*$ when $p\geq 2$. We refer the reader to \cite{LRS} for an excellent overview of various square functions. 

For $1<p\leq 2$, it is known that $G^\alpha$ is bounded on $L^p$ if and only if $\alpha > d(\frac{1}{p}-\frac{1}{2})+\frac{1}{2}$ (see \cite{Su}). On the other hand, in order for $G^\alpha$ to be bounded on $L^p(\R^d)$ for $p>2$, the condition $\alpha > \max(d(\frac{1}{2}-\frac{1}{p}),\frac{1}{2})$ is necessary, and is conjectured to be sufficient. The conjecture for $d=2$ was verified by Carbery \cite{CarM, CarW}, yielding the $L^4$ bound for $S^\alpha_*$ as a corollary. For higher dimensions, the conjecture has been verified for $p>\frac{2(d+2)}{d}$ in \cite{LRSi} (See also \cite{C, Se}). Furthermore, $L^{p,2} \to L^{p}$ endpoint estimates for the critical $\alpha = d(\frac{1}{2}-\frac{1}{p})$ for a smaller $p$-range, $p>\frac{2(d+1)}{d-1}$, were obtained in \cite{LRS}, where $L^{p,q}$ denotes the Lorentz space.

We show that, on the subspace of radial functions, the endpoint estimate is valid for the conjectured $p$-range.
\begin{thm} \label{thm:Euc} Let $d\geq2$, $\frac{2d}{d-1} < p<\infty$ and $\alpha = d(\frac{1}{2}-\frac{1}{p})>\frac{1}{2}$. Then 
\begin{equation*}
\norm{G^\alpha f}_{L^p \rad (\R^d) } \leq C \norm{f}_{L^{p,2} \rad (\R^d)}.
\end{equation*}
\end{thm}

This implies a radial version of the conjecture for $G^\alpha$ by real interpolation. As a consequence, one may obtain a new proof of the sharp estimate for radial Fourier multipliers acting on radial functions in terms of Sobolev spaces (see \cite{GaTr}). A much stronger result is known. Garrig\'{o}s and Seeger \cite{GaSe} obtained a necessary and sufficient condition for $L^p_{rad}(\R^d)$ boundedness of $\mathcal{T}_m$ for $1<p<\frac{2d}{d+1}$. We note that our proof of Theorem \ref{thm:Euc} is based on \cite{GaSe}. 

Let $M_m f := \sup_{t>0} |\mathcal{T}_{m(|t\cdot|)} f|$, where we additionally assume that $m$ is compactly supported in $(0,\infty)$. For the range $1<p<\frac{2d}{d+1}$, a necessary and sufficient condition for $L^p_{rad}(\R^d)$ boundedness of $M_m$ is known (see \cite{GaSeM}). By Theorem \ref{thm:Euc}, we may obtain a sharp sufficient condition for $L^p_{rad}$ boundedness of $M_m$ in terms of Sobolev spaces for $2 \leq p < \infty$ (see Corollary \ref{cor:max}). 

Our primary motivation for Theorem \ref{thm:Euc} comes from a more general situation when multipliers and functions are assumed to be multi-radial. Let $n\in \N$ and $\ve{d}=(d_1,\cdots,d_n)\in \N^n$. We say that $f$ is $\ve{d}$-radial if there is a function $f_0$ on $(0,\infty)^n$ such that $f(x_1,\cdots,x_n) = f_0(|x_1|, \cdots,|x_n|)$, where $x_j \in \R^{d_j}$. In this case, we say that $f$ is the $\ve{d}$-radial extension of $f_0$.

In this paper, we are interested in the Fourier multiplier transformation given by a $\ve{d}$-radial multiplier $m$ acting on $\ve{d}$-radial functions. A typical $m$ would be a tensor product of radial multipliers. In that case, one may easily obtain $L^p$ bounds by iteration. Unfortunately, this argument fails for general $m$. Nevertheless, it is easy to iterate Theorem \ref{thm:Euc} to obtain estimates for product square functions. As a consequence, we obtain sharp Marcinkiewicz type multiplier theorems for the $\ve{d}$-radial case. This shall be carried out in the multivariate Hankel multiplier setting, which improves a result of Wr\'{o}bel \cite{Wr} (see Theorem \ref{thm:gtm}). 

Here we state a special case of Theorem \ref{thm:gtm}. Let us denote by $\R^{\vec{d}}$ and ${L^p_{rad}(\R^{\vec{d}})}$ the product space $\R^{d_1} \times \cdots \times \R^{d_n}$ and the subspace of $\ve{d}$-radial functions in $L^p(\R^{\vec{d}})$, respectively. Let $\phi$ be a tensor product of $n$ non-trivial smooth functions supported on the interval $[1,2]$. It would be convenient to define the subspace ${L^2_{loc,\ve{\alpha}}(\R^n)}$ of $L^2_{loc}(\R^n)$ equipped with the norm
\begin{equation*}
\norm{m_0}_{L^2_{loc,\ve{\alpha}}}^2 := \sup_{\vec{t}\in (0,\infty)^n} \int_{\R^n} |\ft_{\R^n}[\phi m_0(\vec{t}\cdot)](\xi)|^2 \prod_{j=1}^n (1+|\xi_j|)^{2\alpha_j} d\xi,
\end{equation*}
where ($\vec{t}\cdot)$ denotes the $n$-parameter dilation $(t_1\cdot,\cdots,t_n\cdot)$.
\begin{thm}
Let $1<p<\infty$ and $\vec{d}=(d_1,\cdots,d_n)$ such that $d_j\geq 2$. Assume that $m$ is the $\vec{d}$-radial extension of a bounded function $m_0\in {L^2_{loc,\ve{\alpha}}(\R^n)}$ for some $\vec{\alpha} = (\alpha_1,\cdots,\alpha_n)$ such that $\alpha_j > \max(\frac{1}{2}, d_j|\frac{1}{p}-\frac{1}{2}|)$ for $1\leq j\leq n$.
Then 
\begin{equation*}
\norm{\ft^{-1} [m \ft f]}_{L^p(\R^{\vec{d}})} \leq C \norm{m_0}_{L^2_{loc,\ve{\alpha}}} \norm{f}_{L^p_{rad}(\R^{\vec{d}})}.
\end{equation*}
\end{thm}

The paper is organized as follows. In Section \ref{sec:Hankel}, we formulate Theorem \ref{thm:Euc} in a slightly more general context in terms of a square function associated with Hankel multipliers. In Section \ref{sec:mul}, we shall extend the results in Section \ref{sec:Hankel} to multivariate Hankel multipliers. We shall include an application to Bochner-Riesz type multipliers. In Section \ref{sec:product}, product square functions are discussed. Pointwise estimates for multivariate Hankel multiplier transformations in terms of the product square functions are obtained, which leads to multiplier theorems. The rest of the paper shall be devoted to the proof of Theorem \ref{thm:squareend}, which is slightly more general than Theorem \ref{thm:Euc}. In Appendix, we give a proof of $L^p$ bounds of a Littlewood-Paley square function considered in this paper.

In what follows, we shall frequently write $A\les B$ if $A \leq C B$ for some universal implicit constant $C$ which may depend on parameters including $n,p,\vec{d}$, and $\vec{\alpha}$.

\section{Hankel multipliers: Single variable case}\label{sec:Hankel}
Consider a radial function $F$ on $\R^d$, such that $F(x) = f(|x|)$ for a function $f$ on $\R_+:=(0,\infty)$. It is well-known that the Fourier transform of $F$ can be expressed by an integral transform of $f$ which involves the Bessel function. Indeed, one has $\ftd[F](\xi) = (2\pi)^d \han f(|\xi|)$. Here, $\han$ is the modified Hankel transform defined by
\begin{equation*}
\han f(s) = \int_0^\infty B_d(sr) f(r) d\mu_d(r),
\end{equation*}
where $B_d(x) = x^{-\frac{d-2}{2}} J_{\frac{d-2}{2}} (x)$, $J_\alpha$ denotes the standard Bessel function of order $\alpha$, and $\mu_d$ is the measure on $\R_+$ given by $d\mu_d(r)=r^{d-1}dr$ (See \cite{StWe}). In what follows, we shall let $d$ be a real number greater than $1$.

The operator $\han$ enjoys many properties analogous to those of the Fourier transform including the inversion formula and Plancherel's theorem. Let $\s(\R_+)$ be the space of (restrictions to $\R_+$ of) even Schwartz functions on $\R$. Then $\han$ is an isomorphism on $\s(\R_+)$ and an isometry of $L^2(\mu_d)$ with $\han^{-1} = \han$. 

We are now ready to define a variant of the square function $G^\alpha$ relevant to Hankel multipliers. We shall work with $H$-valued functions $f$ on $\R_+$, where $H$ is a separable Hilbert space, for an iteration argument to be used later in Section \ref{sec:proofgtm}. We define the square function $\mathcal{G}^\alpha$ by
\begin{equation*} 
\mathcal{G}^\alpha f (r) = \biggl( \int_0^\infty \abs{ R^\alpha_t f(r) }_H^2 \frac{dt}{t} \biggr)^{1/2},
\end{equation*}
where $\han[R^\alpha_t f](\rho) = \frac{\rho}{t}(1-\frac{\rho}{t})^{\alpha-1}_+ \han f(\rho)$ for $\alpha>1/2$. 

For $1<p\leq 2$, $\mathcal{G}^\alpha$ is bounded on $L^p(\mu_d)$ if and only if $\alpha > d(\frac{1}{p}-\frac{1}{2})+\frac{1}{2}$. The proof is essentially contained in the proof of $L^p(\R^d)$ bounds of $G^\alpha$. For $2\leq p<\infty$, one may verify that $\G^\alpha$ is bounded on $L^p(\mu_d)$ only if $\alpha > \alpha(d,p)$, where
\begin{equation*}
\alpha(d,p) := \max \biggl(\frac{1}{2}, d\abs{\frac{1}{p}-\frac{1}{2}}\biggr).
\end{equation*}
This can be done, for instance, by examining its consequences (see e.g. Corollary \ref{cor:gt}). We show that the condition is also sufficient. 
\begin{thm} \label{thm:square} 
Let $d>1$ and $2 \leq p< \infty$. Then
\begin{equation*}
\norm{ \mathcal{G}^\alpha f}_{L^p(\mu_d)} \leq C \norm{f}_{L^p(\mu_d,H)}
\end{equation*}
if and only if $\alpha > \alpha(d,p)$.
\end{thm}

This result is obtained by real interpolation between the $L^2(\mu_d)$ bound for $\alpha>1/2$ and the following endpoint bounds.
\begin{thm} \label{thm:squareend} 
Let $d>1$, $\frac{2d}{d-1}<p<\infty$ and $\alpha = d(\frac{1}{2}-\frac{1}{p})>\frac{1}{2}$. Then
\begin{equation*}
\norm{ \mathcal{G}^\alpha f}_{L^p(\mu_d)} \leq C \norm{f}_{L^{p,2}(\mu_d,H)}.
\end{equation*}
\end{thm}

Theorem \ref{thm:Euc} is an immediate consequence of Theorem \ref{thm:squareend}. Indeed, observe that $G^\alpha F (x) = \mathcal{G}^\alpha f (|x|)$ if $F(x) = f(|x|)$ and that we may identify $L^{p,q}\rad(\R^d)$ with $L^{p,q}(\mu_d)$. In fact, all results to be discussed in the paper on Hankel multipliers $m(\rho)$ with $L^{p,q}(\mu_d)$ norm can be similarly translated into statements on radial Fourier multipliers $m(|\cdot|)$ with $L^{p,q}_{rad}$ norm.

\begin{remark} The Lorentz space $L^{p,2}$ in Theorem \ref{thm:squareend} may not be replaced by $L^{p,q}$ for $q>2$ (see \cite{LRS}). We do not know if $\mathcal{G}^\alpha$ is actually bounded from $L^{p,2}(\mu_d,H)$ to $L^{p,q}(\mu_d)$ for some $q<p$, in particular for $q=2$.
\end{remark}

Next, we shall state multiplier theorems which follow from the square function estimates. Let $m$ be a bounded function on $\R_+$ and $T_m$ be the operator defined by,
\begin{equation*}
\han [T_m f] (\rho) = m(\rho) \han f(\rho).
\end{equation*}

Let $\Phi\in \s(\Rp)$ such that $\Phi(0)=0$, and $\Phi_t$ be a Hankel multiplier transformation defined by $\han[\Phi_t f](\rho) = \Phi(\frac{\rho}{t}) \han f(\rho)$. We define a Littlewood-Paley function $g_\Phi$ by
\begin{equation*} 
g_\Phi f(r) = \biggl(\int_0^\infty |\Phi_t f(r)|_H^2 \frac{dt}{t}\biggr)^{1/2}.
\end{equation*}
Then $\norm{g_\Phi (f)}_{L^p(\mu_d)}$ is comparable to $\norm{f}_{L^p(\mu_d)}$ for $1<p<\infty$. See Appendix.

We shall use a specific $\Phi$ given by $\Phi(\rho) =  \rho \phi(\rho)$, where $\phi$ is a non-trivial smooth function supported on the interval $[1,2]$. Then for $\alpha>1/2$, there is a pointwise estimate similar to \eqref{eqn:pom}
\begin{equation}\label{eqn:pomhan}
g_\Phi[T_m f](r) \leq C \sup_{t>0} \norm{m(t\cdot) \phi}_{L^2_\alpha(\R)} \G^\alpha f(r).
\end{equation}
See Section \ref{sec:RL}. Thus, we obtain the following sharp multiplier theorem in terms of localized $L^2$ Sobolev spaces.
\begin{cor}\label{cor:gt} Let $d>1$, $1<p<\infty$, and $\phi$ be a non-trivial smooth function supported on $[1,2]$. Suppose that $\sup_{t>0} \norm{m(t\cdot) \phi}_{L^2_\alpha(\R)} < \infty$ for some $\alpha>\alpha(d,p)$. Then the operator $T_m$ is bounded on $L^p(\mu_d)$.
\end{cor}

As was discussed in Introduction, Corollary \ref{cor:gt} is not new. See also \cite{GoSt, DzPr} for multiplier theorems on $L^1$ and Hardy spaces.

Next, we turn to the maximal operators $M_{m} f := \sup_{t>0} |T_{m(t\cdot)} f|$ for a multiplier $m$ supported in $[1/2,2]$. From the square function estimate, we have $L^p$ bounds for the maximal operators $M_m$ for the range $p\geq 2$. 
\begin{cor} \label{cor:max} Let $d>1$ and $2\leq  p<\infty$. Suppose that $m$ is supported in $[\frac{1}{2},2]$ and $m\in L^2_\alpha(\R)$ for $\alpha> \alpha(d,p)$. Then
\begin{equation*}
\norm{M_m f}_{L^p(\mu_d)} \leq C \norm{m}_{L^2_\alpha(\R)}
\norm{f}_{L^{p}(\mu_d)}. 
\end{equation*}
\end{cor}
This is a consequence of a pointwise estimate
\begin{equation} \label{eqn:maxpoint}
M_m f (r) \leq C \norm{m}_{L^2_{\alpha}(\R)} \mathcal{G}^{\alpha} f (r)
\end{equation}
(see Section \ref{sec:RL}) for $\alpha>1/2$.

Corollary \ref{cor:max} is sharp in the sense that the required number of derivative $\alpha(d,p)$ cannot be decreased. This can be seen by considering the truncated Bochner-Riesz multiplier $m(\rho) = (1-\rho^2)_+^\lambda \chi(\rho)$ where $\chi$ is a smooth function supported near $\rho = 1$, discarding a harmless part near the origin. This would also prove $L^p_{rad}$ bounds of $S^\lambda_*$ for $2 \leq p<\frac{2d}{d-1-2\lambda}$, which was previously obtained by Kanjin \cite{Kan}\footnote{In fact, the optimal $p$-range $\frac{2d}{d+1+2\lambda}<p<\frac{2d}{d-1-2\lambda}$ was obtained in \cite{Kan}.}.

While it is sharp, the $L^2$-Sobolev condition is too restrictive to yield endpoint bounds. However, we recently proved that 
\begin{equation*}
\norm{M_m}_{L^{p,q}(\mu_d) \to L^{p}(\mu_d)} \approx \norm{\han m}_{L^{p',q'}(\mu_d)}.
\end{equation*}
for $\frac{2d}{d-1}<p<\infty$ and $1\leq q \leq p$ (see \cite{KimM}), which covers endpoint bounds. See \cite{LRS} for $L^{p,1}(\R^d) \to L^p(\R^d)$ bounds of $S^\lambda_*$ for a smaller $p$-range.

\section{Hankel multipliers: Multivariate case} \label{sec:mul}
The goal of this section is to extend the results of the previous section in
the multivariate setting. Fix $n \in \N$ and $\ve{d}=(d_1,\cdots,d_n)\in \R^n$
such that $d_j \geq 1$. Hankel transform $\h$ acting on functions on $(\R_+)^n$
is defined by 
$\h f (s):= {\hh}_{d_n} \cdots {\hh}_{d_1} f (s)$, where $\hh_{d_k}$ acts only on the $k$-th variable. 

For $\ve{d}\in \N^n$, $\h$ generalizes the Fourier transform of $\ve{d}$-radial functions. Suppose that $\tilde{m}$ is $\ve{d}$-radial extension of a bounded function $m$ on $(\R_+)^n$. Then the study of $\mathcal{T}_{\tilde{m}}$ acting on $\ve{d}$-radial functions is reduced to the study of $T_m$ defined by $\h[T_mf] = m \h f$. 

The operator $T_m$ have been studied only recently (see e.g. \cite{BeCaNo, Wr, DzPrWr}). In particular, Wr\'{o}bel \cite{Wr} proved a Marcinkiewicz type multiplier theorem, where a smoothness condition was given in terms of a variant of $L^2$ Sobolev space. We introduce several notations in order to simplify the presentation.

\begin{notation} 
Let $\mu_{\ve{d}}$ be the measure on $(\R_+)^n$ given by $d\mu_{\ve{d}}(s) = \prod_{k=1}^n d\mu_{d_k}(s_k)$.
For given $\ve{t},\ve{s} \in \R^n$, we define $\ve{t} \ve{s}$ and $\ve{t}/\ve{s}$ the vectors given by component-wise product and division, respectively. We write $\ve{t}>\ve{s}$ if $t_k > s_k$ for all $1\leq k\leq n$. If $s\in \R$, we write $\ve{t}>s$ if $t_k > s$ for all $k$. For $1\leq p < \infty$, let $\ve{\alpha}(\ve{d},p)\in (\Rp)^n$ be the vector whose $k$-th component is $\max\left(\frac{1}{2},d_k |\frac{1}{p} - \frac{1}{2}|\right)$. For a given $\ve{\alpha}\in (\Rp)^n$, we shall denote by $L^2_{\ve{\alpha}}(\R^n)$ the Sobolev space equipped with the norm 
$\norm{f}_{L^2_{\ve{\alpha}}(\R^n)} = \norm{w_{\ve{\alpha}}\ftn[f]}_{L^2(\R^n)}$, where $w_{\ve{\alpha}}(\xi) = \prod_{k=1}^n (1+|\xi_k|)^{\alpha_k}$.
\end{notation}

Let $\{\phi_k\}_{1\leq k\leq n}$ be a collection of non-trivial smooth functions supported on the interval $[1,2]$, and let $\phi(r) = \prod_{k=1}^n \phi_k(r_k)$. It was shown in \cite{Wr} that $T_m$ is bounded on $L^p(\mu_{\ve{d}})$ for $1<p<\infty$ if
\begin{equation}\label{eqn:condition}
\sup_{\ve{t}>0} \norm{m(\ve{t}\cdot) \phi}_{L^2_{\ve{\alpha}}(\R^n)} < \infty
\end{equation}
for some $\ve{\alpha} > \ve{\alpha}(\ve{d},1)$. Here we have used the notation $(\ve{t}\cdot) = (t_1\cdot,\cdots,t_n\cdot)$ for the $n$-parameter dilation.

By using a product version of Theorem \ref{thm:square} (see Theorem \ref{thm:squareprod}), we may improve the result as follows.
\begin{thm}\label{thm:gtm} Let $\ve{d} > 1$ and $1<p<\infty$. Suppose that \eqref{eqn:condition} holds for some $\ve{\alpha}>\ve{\alpha}(\ve{d},p)$. Then the operator $T_m$ is bounded on $L^p(\mu_{\ve{d}})$.
\end{thm}
This is sharp in the sense that $\ve{\alpha}(\ve{d},p)$ cannot be decreased. One may verify this from the sharpness of Corollary \ref{cor:gt} by considering product type multipliers. Theorem \ref{thm:gtm} implies the following.

\begin{cor} Let $1<p<\infty$, $\ve{\alpha}\in \Z^n$, and $\ve{\alpha}>\ve{\alpha}(\ve{d},p)$. Suppose that 
\begin{equation} \label{eqn:dercon}
\sup_{\ve{t}>0} \int_{t_n}^{2t_n} \cdots \int_{t_1}^{2t_1} |\rho^{\ve{\beta}} D^{\ve{\beta}} m(\rho)|^2 \frac{d\rho_1}{t_1} \cdots \frac{d\rho_n}{t_n} < \infty
\end{equation}
for all $\ve{\beta} \in \Z^n$, $0\leq \ve{\beta}\leq \ve{\alpha}$. Then $T_m$ is bounded on $L^p(\mu_{\ve{d}})$. In particular, \eqref{eqn:dercon} holds if $|D^{\ve{\beta}} m (\rho)| \leq C\rho^{-\ve{\beta}}$ for all $ 0 \leq \ve{\beta} \leq \ve{\alpha}$.
\end{cor}

We may also extend Corollary \ref{cor:max} for the $n$-parameter maximal operator $M_m f := \sup_{\ve{t} >0} |T_{m(\ve{t} \cdot)} f|$.
\begin{thm}\label{thm:maxmul}
Let $\ve{d} >1$ and $2 \leq p<\infty$. Suppose that $m$ is supported
in $[\frac{1}{2},2]^n$ and $m \in L^2_{\ve{\alpha}}(\R^n)$ for some
$\ve{\alpha}>\ve{\alpha}(\ve{d},p)$. Then 
\begin{equation*}
\norm{M_m f}_{L^p(\mu_{\ve{d}})} \leq C \norm{m}_{L^2_{\ve{\alpha}}(\R^n)}
\norm{f}_{L^{p}(\mu_{\ve{d}})}. 
\end{equation*}
\end{thm}

See Section \ref{sec:product} for the proof of Theorem \ref{thm:gtm} and \ref{thm:maxmul}.

\subsection{Application to Bochner-Riesz type multipliers} \label{sec:Boch}
Let $m^\lambda(\rho) = (1-|\rho|^2)^\lambda_+$, where $|\rho|^2 = \rho_1^2 + \cdots + \rho_n^2$, and $T^\lambda$ be the operator defined by $\h [T^\lambda f] = m^\lambda \h f$. Let us temporarily assume that $\ve{d} \in \N^n$. Then the study of the usual Bochner-Riesz means for the Fourier transform acting on $\ve{d}$-radial functions reduces to the study of $T^\lambda$ since
\begin{equation*}
S^\lambda F(x_1,\cdots,x_n) = T^\lambda f (|x_1|,\cdots,|x_n|)
\end{equation*}
if $F$ is the $\ve{d}$-radial extension of $f$. Note that $T^\lambda$ cannot be bounded on $L^p(\mu_{\ve{d}})$ unless $|\frac{1}{p} -\frac{1}{2}| < \frac{1}{\norm{d}} \left(\lambda+\frac{1}{2}\right)$, where $\norm{d} = \sum_{k=1}^n d_k$.
\begin{cor}  \label{thm:boch}
Let $\ve{d}>1$ and $\lambda > \max_{1\leq k\leq n} \frac{\norm{d}}{2d_k} -\frac{1}{2}$. Then $T^\lambda$ is bounded on $L^p(\mu_{\ve{d}})$ if 
\begin{equation} \label{eqn:boch}
\abs{\frac{1}{p} -\frac{1}{2}} < \frac{1}{\norm{d}} \left(\lambda+\frac{1}{2}\right).
\end{equation}
\end{cor}

This follows from Theorem \ref{thm:gtm}. Since Theorem \ref{thm:gtm} applies to general multipliers, not necessarily radial, one does not expect a sharp result for $m^\lambda$ from Theorem \ref{thm:gtm} for $n>1$. However, we expect that one may improve the statement by taking advantage of $m^\lambda$ being radial. We hope to return to this issue in future work.

\begin{lem} \label{lem:boch} Let $\phi$ be a tensor product of $n$ smooth functions supported on $[1,2]$. If $0\leq \beta < \lambda + \frac{1}{2}$, then
\begin{equation*}
\sup_{\ve{t}>0} \norm{ m^\lambda (\ve{t} \cdot) \phi }_{L^2_\beta(\R^n)}  < \infty.
\end{equation*} 
\end{lem}

Let us show how Corollary \ref{thm:boch} follows from Lemma \ref{lem:boch}. Assume \eqref{eqn:boch}. Then there is $\epsilon(p)>0$ such that if we define
\begin{equation*}
\alpha_k = \frac{d_k}{\norm{d}} \left(\lambda+\frac{1}{2}\right) -\epsilon(p)
\end{equation*}
for $1 \leq k \leq n$, then $\ve{\alpha} > \ve{\alpha}(\ve{d},p)$ and $\norm{\ve{\alpha}} := \sum_{j=1}^n \alpha_j < \lambda  + \frac{1}{2}$. By Theorem \ref{thm:gtm} together with Lemma \ref{lem:boch} and the trivial embedding $L^2_{\norm{\ve{\alpha}}} \hookrightarrow L^2_{\ve{\alpha}}$, we have Corollary \ref{thm:boch}.

\begin{proof} [Proof of Lemma \ref{lem:boch}]
Although the proof seems to be standard, we include it for completeness. We apply the standard dyadic decomposition for the Bochner-Riesz multipliers.
Take a smooth function $\chi$ supported on $\frac{1}{2} \leq x \leq 2$, such that $\sum_{l=0}^\infty \chi(2^l x) = 1$ if $0<x\leq 1$. Then one may write $m^\lambda (\rho) = \sum_{l=0}^\infty 2^{-l\lambda} m^\lambda_l(\rho),$ where 
\begin{equation*}
m^\lambda_l(\rho) = 2^{l\lambda} (1-|\rho|^2)^\lambda \chi(2^l (1-|\rho|^2)).
\end{equation*}

Then we have
\begin{equation*}
\sup_{\ve{t}>0} \norm{ m^\lambda (\ve{t} \cdot) \phi }_{L^2_\beta(\R^n)}  \leq \sum_{l=0}^\infty 2^{-l\lambda} \sup_{\ve{t}>0} \norm{ m^\lambda_l (\ve{t} \cdot) \phi }_{L^2_\beta(\R^n)}.
\end{equation*}
Since $m^\lambda_l (\ve{t} \cdot) \phi \equiv 0$ if $\abs{\ve{t}} >1$, we may assume that the supremum is taken over $|\ve{t}|\leq 1$. 
For $l\leq 2$, it is easy to show that 
\begin{equation*}
\sup_{\ve{t}>0} \norm{ m^\lambda_l (\ve{t} \cdot) \phi }_{L^2_\beta(\R^n)} < \infty
\end{equation*}
for any $\beta \geq 0$.

For $l>2$, we need to show that 
\begin{equation} \label{eqn:sob}
\sup_{\ve{t}>0} \norm{ m^\lambda_l (\ve{t} \cdot) \phi }_{L^2_\beta(\R^n)} \leq C 2^{l(\beta-\frac{1}{2})}.
\end{equation}
Here, we may further assume that $1/4 \leq  |\ve{t}| \leq  1$ since $m^\lambda_l$ is compactly supported away from the origin. Moreover, by interpolation, it is enough to show \eqref{eqn:sob} for integer $\beta \geq 0$.

By a direct calculation, for $\norm{\ve{\gamma}} \leq \beta$ and $|\ve{t}| \leq 1$,
\begin{equation*}
|D^{\ve{\gamma}} [m^\lambda_l(\ve{t} \cdot) \phi ]|(\rho) \les 2^{l\beta} \tilde{\chi}(2^l(1-|\ve{t} \cdot \rho|^2)) \tilde{\phi}(\rho),
\end{equation*}
where $\tilde{\chi}$ and $\tilde{\phi}$ are finite sums of derivatives of $\chi$ and $\phi$, respectively.

Thus, \eqref{eqn:sob} follows from 
\begin{equation} \label{eqn:sob2}
\sup_{1/4 \leq |\ve{t}| \leq 1} \int |\tilde{\chi}(2^l(1-|\ve{t} \cdot \rho|^2)) \tilde{\phi}(\rho)|^2 d\rho \les 2^{-l}.
\end{equation}

The integral on the left hand side of \eqref{eqn:sob2} is bounded by 
\begin{align*}
\int_{t_n}^{2t_n} \cdots \int_{t_1}^{2t_1} |\tilde{\chi}(2^l(1-|\rho|^2))|^2 \frac{d\rho_1}{t_1} \cdots \frac{d\rho_n}{t_n}.
\end{align*}

Thus, we are led to consider the volume of the intersection between the box $\prod_{j=1}^n [t_j,2t_j]$ and an annulus of radius and width comparable to $1$ and $2^{-l}$, respectively. Assume, without the loss of generality, that $t_1 \geq t_2 \geq \cdots \geq t_n$. Then $t_1 \geq c_n$ for some $c_n>0$ since $|\ve{t}|>1/4$. 

We claim that
\begin{equation*}
\int_{t_1}^{2t_1} |\tilde{\chi}(2^l(1-|\rho|^2))|^2 \frac{ d\rho_1}{t_1} \les 2^{-l}
\end{equation*}
provided that $t_1\geq c_n$, which would imply \eqref{eqn:sob2}. When evaluating the integral, we may assume that $|\rho'|^2 \leq 1-c_n^2 < 1$ since $\rho_1\geq c_n$, where $\rho = (\rho_1,\rho')$. The claim follows from the fact that the $\rho_1$ support of $\chi(2^l(1-|\rho|^2))$ for each fixed $\rho'$ with $|\rho'|^2 \leq 1-c_n^2$, is contained in an interval of size $O(2^{-l})$. This proves the claim, and thus \eqref{eqn:sob2}.
\end{proof}

\section{Product variants}\label{sec:product}
\subsection{Product square functions} \label{sec:proofgtm}
Let $\Phi^{(k)}(\rho_k) = \rho_k \phi_k(\rho_k)$ for $\phi_k$ as in Section \ref{sec:mul}, $\Phi_{\ve{t}} f = \Phi^{(n)}_{t_n} \cdots \Phi^{(1)}_{t_1} f$, and $R^{\ve{\alpha}}_{\ve{t}}f = R^{\alpha_n}_{t_n} \cdots R^{\alpha_1}_{t_1} f$, where $\Phi^{(k)}_{t_k}$and $R^{\alpha_k}_{t_k}$ act only on $k$-th variable. For $H$-valued functions $f$, we define $\G^{{\ve{\alpha}}}$ by
\begin{equation}\label{def:sq} 
\G^{{\ve{\alpha}}} f(r) = \biggl( \int_{(\Rp)^n} \abs{ R^{{\ve{\alpha}}}_{\ve{t}}  f (r)}_H^2 \frac{d\ve{t}}{\ve{t}} \biggr)^{1/2},
\end{equation}
where $\frac{d\ve{t}}{\ve{t}}$ is the measure $\prod_{k=1}^n \frac{dt_k}{t_k}$ on $(\Rp)^n$. $g_\Phi$ is similarly defined as in \eqref{def:sq}, with $\Phi_{\ve{t}}$ in place of $R^{{\ve{\alpha}}}_{\ve{t}}$.

We first note that for $1<p<\infty$, there is a constant $C>0$ such that
\begin{equation*}
C^{-1} \norm{f}_{L^p(\mu_{\ve{d}},H)} \leq \norm{g_\Phi f}_{L^p(\mu_{\ve{d}})} \leq C \norm{f}_{L^p(\mu_{\ve{d}},H)}.
\end{equation*}
The second inequality follows from the case $n=1$ (see Appendix) by an iteration argument (See \cite[Section 2]{FeSt}). The first inequality follows from the second inequality by the polarization identity and $\norm{g_\Phi f}_{L^2(\mu_{\ve{d}})} = C \norm{f}_{L^2(\mu_{\ve{d}},H)}$.

There are product versions of the pointwise estimates \eqref{eqn:pomhan} and \eqref{eqn:maxpoint}. For ${\ve{\alpha}}>1/2$, we have
\begin{align}\label{eqn:pomprod}
g[T_m f](r) &\leq C \sup_{\ve{t} >0 } \norm{m(\ve{t}\cdot) \phi}_{L^2_{\ve{\alpha}}(\R^n)} \G^{\ve{\alpha}} f(r), \\
\label{eqn:pmaxprod}
M_m f (r) &\leq C \norm{m}_{L^2_{\alpha}(\R)} \G^{\ve{\alpha}} f(r),
\end{align}
where we additionally assume that $m$ is supported in $[1/2,2]^n$ in \eqref{eqn:pmaxprod}. We defer the proof of the estimates in the following section.

Given the pointwise estimates, Theorem \ref{thm:gtm} and \ref{thm:maxmul} are consequences of the following theorem.
\begin{thm} \label{thm:squareprod} 
Let $\ve{d}>1$ and $2 \leq p< \infty$. Then
\begin{equation*}
\norm{ \mathcal{G}^{\ve{\alpha}} f}_{L^p(\mu_{\ve{d}})} \leq C \norm{f}_{L^p(\mu_{\ve{d}},H)}
\end{equation*}
if and only if ${\ve{\alpha}} > \ve{\alpha}(\ve{d},p)$.
\end{thm}
\begin{proof}
For the necessary condition, it is enough to consider a function $f(r)=\prod_{1\leq k \leq n} f_k(r_k)$ such that $f_k \in L^p(\mu_{d_k},H)$. Then $\G^{\ve{\alpha}} f(r) = \prod_{k=1}^n \G^{{\alpha}_k} f_k(r_k)$, thus the necessary condition follows from Theorem \ref{thm:square}.

An iteration argument in \cite{FeSt} gives the sufficient condition, but we shall include the argument for the convenience of the reader. We shall use induction on $n$ with the case $n=1$ given by Theorem \ref{thm:square}. Suppose that the theorem is true for the dimension $n-1$. Let ${\ve{\alpha}} = (\ve{\alpha}',\alpha_n)\in \R^{n-1} \times \R$ such that ${\ve{\alpha}} > \ve{\alpha}(\ve{d},p)$. Set $F(r',r_n) = R^{\alpha_n}_{t_n}[f(r', \cdot)] (r_n)$. We regard $F$ as a $\tilde{H}$-valued function, where $\tilde{H} = L^2(\Rp, \frac{dt_n}{t_n}, H)$ is a Hilbert space. 

We have $|F(r',r_n)|_{\tilde{H}} = \G^{\alpha_n} [f(r',\cdot)](r_n)$ and $\G^{\ve{\alpha}} f(r) = \G^{\ve{\alpha}'}[F(\cdot,r_n)](r')$. Thus, $\norm{\G^{\ve{\alpha}} f}_{L^p(\mu_{\ve{d}})}^p$ is equal to
\begin{align*} 
&\int_0^\infty \int_{(\Rp)^{n-1}} |\G^{\ve{\alpha}'} [F(\cdot,r_n)](r')|^p d\mu_{{\ve{d}}'}(r') d\mu_{d_n}(r_n) \\
&\les \int_0^\infty \int_{(\Rp)^{n-1}} |F(r',r_n)|_{\tilde{H}}^p d\mu_{{\ve{d}}'}(r') d\mu_{d_n}(r_n) \\
&= \int_{(\Rp)^{n-1}} \int_0^\infty |\G^{\alpha_n} [f(r',\cdot)](r_n)|^p d\mu_{d_n}(r_n) d\mu_{{\ve{d}}'}(r') \\
&\les \int_{(\Rp)^{n-1}} \int_0^\infty |f(r',r_n)|_H^p d\mu_{d_n}(r_n) d\mu_{{\ve{d}}'}(r') = C \norm{f}_{L^p(\mu_{\ve{d}},H)}^p,
\end{align*}
where the first inequality follows from the induction hypothesis.
\end{proof}

\subsection{Pointwise estimates} \label{sec:RL}
In this section, we prove \eqref{eqn:pomprod} and \eqref{eqn:pmaxprod}. The Riemann-Liouville lemma on fractional differentiation played an important role for the pointwise estimate \eqref{eqn:pom} given in \cite{CarR}. We shall need a product version of the Riemann-Liouville lemma.

Let ${\ve{\alpha}} = (\alpha_1, \cdots, \alpha_n)$ and ${\ve{\alpha}} \geq 0$. We define the fractional differentiation $\fd$ for $f\in L^2(\R^n)$ by $$\ft\left[ D^{{\ve{\alpha}}} f \right](\xi) = \prod_{k=1}^n (-i\xi_k)^{\alpha_k} \ft{f}(\xi),$$
which coincides with the usual differentiation up to constant when ${\ve{\alpha}} \in \Z^n$. In addition, let $I^{\ve{\alpha}}$ be the fractional integral defined by 
\begin{equation*}
I^{\ve{\alpha}}f (x)=  \int_{\R^n} \prod_{k=1}^n \frac{(u_k-x_k)_+^{\alpha_k-1} }{\Gamma(\alpha_k)} f(u)du.
\end{equation*}

We shall work with the Sobolev space $L^2_{\ve{\alpha}}(\R^n)$ defined in Section \ref{sec:mul}. 
\begin{lem} [Riemann-Liouville] \label{lem:RiLi} Let ${\ve{\alpha}}>1/2$. Suppose that $f \in L_{\ve{\alpha}}^2(\R^n)$ and that $\supp f \subset \prod_{k=1}^n (-\infty,a_k]$. Then $\supp \fd f  \subset  \prod_{k=1}^n (-\infty,a_k]$ and 
\begin{equation*}
 f(x) = I^{\ve{\alpha}} [\fd f](x), \text{ a.e. }
\end{equation*}
\end{lem}
\begin{proof}
A proof of the statement for $n=1$ was given in \cite{CarR}. The proof continues to work for the case $n>1$ with a few minor changes. For the convenience of the reader, we shall include the proof for $n=1$ and indicate the changes for $n>1$. 

First consider the operators $D^{\alpha}_\epsilon$ associated with the Fourier multipliers $(\epsilon-i\xi)^{\alpha}$ for $\epsilon>0$. Then $D^{\alpha}_\epsilon f \to D^{\alpha} f$ in $L^2$ as $\epsilon \to 0$ by the dominated convergence theorem.

With an aid of Cauchy's theorem, one may verify that 
\begin{equation*}
g_\epsilon^{\alpha}(x) = \ift[(\epsilon-i\xi)^{-{\alpha}}](x) = \frac{1}{\Gamma(\alpha)} (-x)_+^{\alpha-1} e^{\epsilon x} \in L^1.
\end{equation*}
See \eqref{eqn:kappa} for a related calculation. Thus, the convolution operator $I^\alpha_\epsilon f = f*g_\epsilon^\alpha$ is the inverse of $D^\alpha_\epsilon$ and $I^\alpha_\epsilon f \in L^2$. 

For the support condition, observe that 
\begin{equation*}
D^{\alpha}_\epsilon f(x) = D_\epsilon^{[\alpha]+1} h(x),
\end{equation*}
where $h= I_{\epsilon}^{[\alpha]+1-\alpha} f$ and $[\alpha]$ is the least integer such that $[\alpha] \geq \alpha$. Then the support of $h$ is contained in $(-\infty,a]$. Moreover, $D^{[\alpha]+1}_\epsilon h$ is a linear combination of (non-fractional) derivatives of $h$, preserving the support of $h$.

Next, for a fixed $x\in (-\infty,a)$, one estimates $|f(x) - I^{\alpha} ( D^\alpha f)(x)|$ by
\begin{align*}
 &\frac{1}{\Gamma(\alpha)} \int_x^a (u-x)^{\alpha-1} \abs{ e^{-\epsilon (u-x)} D^\alpha_\epsilon f(t) - D^\alpha f(u) } du \\
&\leq \frac{1}{\Gamma(\alpha)} \norm{(\cdot-x)^{\alpha-1}}_{L^2(I_x)} \norm{e^{-\epsilon (\cdot-x)} D^\alpha_\epsilon f - D^\alpha f}_{L^2(I_x)},
\end{align*}
where $I_x = (x,a)$. The first norm is finite and the second norm tends to $0$ as $\epsilon \to 0$.

The proof easily extends to the case $n>1$. Indeed, one may extend $D^{\alpha}_\epsilon$ by $D^{\vec{\alpha}}_\epsilon$ which is associated with the Fourier multipliers $\prod_{k=1}^n (\epsilon-i\xi_k)^{\alpha_k}$ and make similar multi-parameter extensions.
\end{proof}

With Lemma \ref{lem:RiLi}, product versions of the pointwise estimates given in \cite{CarR} can be obtained without much additional work. We conclude this section with proofs of \eqref{eqn:pomprod} and \eqref{eqn:pmaxprod}.
\begin{proof}[Proof of \eqref{eqn:pomprod}]
We may assume that $\norm{m(\ve{t}\cdot)\phi}_{L^2_{\ve{\alpha}}(\R^n)}<\infty$ for each $\ve{t}$. Fix $\ve{t} \in (\Rp)^n$, and let $h(\rho) = m(\ve{t}\rho)\phi(\rho)$. Then by the Riemann-Liouville Lemma, the support of $D^{\vec{\alpha}} h$ is contained in $(-\infty,2]^n$. Moreover,
\begin{align*}
&\h [\Phi_{\ve{t}} [T_mf] ](\rho) = \prod_{k=1}^n (\rho_k/t_k) \phi_k(\rho_k/t_k) m(\rho) \h f(\rho) \\
&= C_{\ve{\alpha}} \int_{\R^n} \prod_{k=1}^n \frac{\rho_k}{t_k} \biggl(u_k-\frac{\rho_k}{t_k}\biggr)_+^{\alpha_k -1} \h f(\rho)  D^{\vec{\alpha}} h(u)du.\\
&= C_{\ve{\alpha}} \int_{[0,2]^n} \prod_{k=1}^n u_k^{\alpha_k} \frac{\rho_k}{u_k t_k} \biggl(1-\frac{\rho_k}{u_kt_k}\biggr)_+^{\alpha_k -1} \h f(\rho)  D^{\vec{\alpha}} h(u)du
\end{align*}
for a.e. $\rho\in (\Rp)^n$. Applying $\h$ to both sides of the equality,
\begin{align*}
|\Phi_{\ve{t}} [T_mf] (r)|_H &= C_{\ve{\alpha}} \abs{\int_{[0,2]^n} \prod_{k=1}^n u_k^{\alpha_k} R^{\ve{\alpha}}_{\ve{t}u} f(r) D^{\vec{\alpha}} h(u)du}_H. \\
&\les \sup_{{\ve{t}}>0} \norm{m({\ve{t}}\cdot)\phi}_{L^2_{\ve{\alpha}}(\R^n)} \biggl(\int_{[0,2]^n} |R^{\ve{\alpha}}_{\ve{t}u} f(r)|_H^2 du\biggr)^{1/2}.
\end{align*}

The proof is complete after taking the $L^2((\Rp)^n, \frac{d\ve{t}}{\ve{t}})$ norm followed by Fubini's theorem and a change of variable.
\end{proof}

\begin{proof}[Proof of \eqref{eqn:pmaxprod}]
First observe that one may write
\begin{equation*}
\frac{m(\rho)}{\rho_1\cdots \rho_n} = m(\rho) \chi(\rho)
\end{equation*}
for a smooth function $\chi$ supported in $[1/4,4]^n$ if $m$ supported in $[1/2,2]^n$.

We apply the product version of Riemann-Liouville lemma to the function $m\chi$ without the use of the square function $g$. Arguing as in the proof of \eqref{eqn:pomprod}, we obtain
\begin{align*}
|T_{m(\cdot/\ve{t})}f(r)| &= C_{\ve{\alpha}} \abs{\int_{(\Rp)^n} R^{\ve{\alpha}}_{\ve{t}u}f(r) u^{{\ve{\alpha}}+\ve{1}} D^{\vec{\alpha}} [m\chi](u)\frac{du}{u}} \\
&\les \mathcal{G^{\ve{\alpha}}}f(r) \biggl( \int_{(\Rp)^n} \abs{u^{{\ve{\alpha}}+\ve{1}} D^{\vec{\alpha}} [m\chi](u)}^2\frac{du}{u}\biggr)^{1/2} \\
&\les \mathcal{G^{\ve{\alpha}}}f(r) \norm{m\chi}_{L^2_{\ve{\alpha}}(\R^n)},
\end{align*}
where $\ve{1}=(1,\cdots,1)$, $\frac{du}{u}= \prod_{k=1}^n \frac{du_k}{u_k}$, and we have used the fact that the support of the function $D^{\vec{\alpha}} [m\chi]$ is contained in $(-\infty,2]^n$. Finally, we have $\norm{m\chi}_{L^2_{\ve{\alpha}}(\R^n)} \les \norm{m}_{L^2_{\ve{\alpha}}(\R^n)}$ since $\hat{\chi}$ is a Schwartz function. This completes the proof.
\end{proof}

\section{Reductions toward Theorem \ref{thm:squareend}}\label{sec:First}
\subsection{Littlewood-Paley theory} \label{sec:LiPa}
Split the multiplier of $R^\alpha$ as 
$$\rho(1-\rho)_+^{\alpha-1} = \rho \chi_0(\rho) + (1-\rho)_+^{\alpha-1} \chi(\rho),$$
where $\chi_0, \chi\in \s(R_+)$ and $\chi_0$ and $\chi$ are supported on $[0,3/4]$ and $[1/2,2]$, respectively. Then $g_{\Phi}$ with $\Phi(\rho)=\rho \chi_0(\rho)$ is a standard Littlewood-Paley function, which is bounded on $L^p(\mu_d)$ for $1<p<\infty$. Thus, we may assume that 
\begin{equation*}
\han[R^\alpha_t f](\rho) = \biggl(1-\frac{\rho}{t}\biggr)_+^{\alpha-1} \chi\biggl(\frac{\rho}{t}\biggr) \han f(\rho).
\end{equation*}

Using this reduction and Littlewood-Paley theory, we may localize the $t$-integral on the interval $[1,2]$. Choose a cut off function $\eta \in C_0^\infty(\R^+)$ supported on $(1/8,8)$ such that $\eta(\rho) = 1$ on $[1/4,4]$ and define the Littlewood-Paley projection $L_j$ by $\han[L_j f] (\rho) = \eta(2^{-j} \rho) \han{f}(\rho)$. We have
\begin{align*}
[\G^\alpha f (r)]^2 &= \int_0^\infty |R^\alpha_t  f (r) |^2_H \frac{dt}{t} = \sum_j \int_{2^j}^{2^{j+1}} |R^\alpha_t  f (r) |^2_H \frac{dt}{t} \\
&= \sum_j \int_{1}^{2} |R^\alpha_{2^jt}  f (r) |^2_H \frac{dt}{t} =\sum_j \int_{1}^{2} |R^\alpha_{2^jt}  [L_j f] (r) |^2_H \frac{dt}{t},
\end{align*}
where the last equality follows from $\eta(\rho) \chi(\rho/t) = \chi(\rho/t)$ for $t\in [1,2]$.
Thus by the Littlewood-Paley inequality (a discrete version of Theorem \ref{thm:gfuncH}), Theorem \ref{thm:squareend} follows from
\begin{equation} \label{eqn:square0}
\bnorm{\biggl( \sum_j \int_1^2 |R^\alpha_{2^jt}  f_j |^2_H \frac{dt}{t} \biggr)^{1/2} }_{L^p(\mu_d)} \les
\bnorm{\biggl( \sum_j |f_j|_H^2 \biggr)^{1/2} }_{L^{p,2}(\mu_d)}
\end{equation}
for $p = 2d/(d-2\alpha)$ and $1/2 < \alpha < d/2$.

\subsection{Dualization}
Let us denote by $L^2_t$ the Hilbert space $L^2([1,2], dt/t)$.
If $g_j$ is a function which takes values in $L^2_t(H^*)$, namely that $(g_{j})_t(r) := [g_j(r)](t) \in H^*$, then
\begin{equation*}
\int_0^\infty \int_1^2 \inn{R^\alpha_{2^jt}  f_j(r), (g_{j})_t(r)}\frac{dt}{t} d\mu_d(r)  = \int_0^\infty \inn{ f_j(s), \mathcal{R}^\alpha_j g_j(s) } d\mu_d(s),
\end{equation*}
where 
\begin{equation*}
\mathcal{R}^\alpha_j g(s) = \int_1^2 R^\alpha_{2^jt}  g_t(s) \frac{dt}{t}
\end{equation*}
for $L^2_t(H^*)$-valued functions $g$. Thus by duality, (\ref{eqn:square0}) is equivalent to 

\begin{equation} \label{eqn:square}
\bnorm{\biggl( \sum_j |\mathcal{R}^\alpha_j g_j |^2_{H} \biggr)^{1/2} }_{L^{p,2}(\mu_d)} \les
\bnorm{\biggl( \sum_j |g_j|_{L^2_t(H)}^2 \biggr)^{1/2} }_{L^{p}(\mu_d)}
\end{equation}
for $p = 2d/(d+2\alpha)$ and $1/2 < \alpha < d/2$.

\subsection{Decomposition}
We make a dyadic decomposition following \cite{GaSe}. For $m\in \Z$, we let $I_m = [2^m,2^{m+1})$, $I_m^* = [2^{m-1},2^{m+2})$, ${I_m^{**}} = [2^{m-2},2^{m+3})$, $L_{m} = (0,2^{m})$, and $R_{m}=[2^{m},\infty)$. Then we may write
\begin{align*}
\mathcal{R}^\alpha_j f &= \sum_m [ E^\alpha_{j,m} f + S^\alpha_{j,m} f+ V^\alpha_{j,m}f ] = \sum_m E^\alpha_{j,m-j} f + \sum_m S^\alpha_{j,m} f + \sum_m  V^\alpha_{j,m-j}f,
\end{align*}
where $E^\alpha_{j,m} f$, $S^\alpha_{j,m} f$ and $V^\alpha_{j,m} f$ are defined by $\mathcal{R}^\alpha_j(f \chi_{I_m})$ times the characteristic functions  $\chi_{L_{m-2}}$, $\chi_{I_m^{**}}$ and $\chi_{R_{m+3}}$, respectively.

We shall prove the following propositions in Section \ref{sec:proof}.
\begin{prop} \label{prop:HE}
Let $1/2<\alpha<d/2$ and $p = 2d/(d+2\alpha)$. Then there is a constant $\delta(p) > 0$ such that
\begin{equation*}
\norm{V^\alpha_{j,m-j}f}_{\lpl} \leq C 2^{-|m|\delta(p)} \norm{f\chi_{I_{m-j}}}_{L^{p,\infty}(\mu_d,L^2_t(H))},
\end{equation*}
where the constant $C$ does not depend on $j$ and $m$.
\end{prop}

\begin{prop} \label{prop:E}
Let $\alpha>1/2$ and $1\leq p \leq 2$. Then there is a constant $\delta > 0$ such that
\begin{equation*}
\norm{E^\alpha_{j,m-j}f}_{\lp} \leq C 2^{-|m|\delta} \norm{f\chi_{I_{m-j}}}_{\lpt},
\end{equation*}
where the constant $C$ does not depend on $j$ and $m$.
In fact, one may take $\delta=\min( \alpha-1/2,d)$.
\end{prop}

We shall use vector notations $\ve{f} = \{ f_j \}_j$ and $\ve{S}^\alpha_{m} \ve{f} = \{ S^\alpha_{j,m} f_j \}_j$.
\begin{prop} \label{prop:S}
Let $\alpha>1/2$. For each $1 < p \leq 2$, 
\begin{align*}
\bbnorm{\ve{S}^\alpha_{m} \ve{f} }_{L^p( \mu_d,l^2(H))} \leq C \bbnorm{\ve{f} \chi_{I_m}}_{L^p(\mu_d,l^2(L_t(H)))}
\end{align*}
with a constant $C$ which does not depend on $m$.
\end{prop}

\begin{proof}[Proof of (\ref{eqn:square}) given Proposition \ref{prop:HE}, \ref{prop:E} and \ref{prop:S}]
The proof to be given is just a minor modification of the proof given in \cite{GaSe}, but we include it for the convenience of the reader. We show (\ref{eqn:square}) for $\sum_m  V^\alpha_{j,m-j}$ first. 
\begin{align*}
&\bnorm{\biggl( \sum_j \babs{\sum_m V^\alpha_{j,m-j} f_j }^2_{H} \biggr)^{1/2} }_{L^{p,2}(\mu_d)}
\leq \sum_m \pare{ \sum_j \norm{V^\alpha_{j,m-j} f_j}_{\lpl}^p}{p} \\
&\les \sum_m 2^{-|m|\delta(p)}  \pare{ \int_0^\infty \sum_j \chi_{I_{m-j}}(r) |f_j(r)|^p_{L^2_t(H)}d\mu_d(r)}{p},
\end{align*}
where we have used Minkowski's inequality to pull out the sum in $m$ and that $l^p(\Z) \subset l^2(\Z)$. To be more precise, \cite[Lemma 2.1]{GaSe} was used in order to deal with the Lorentz space $L^{p,2}$. Next, the trivial bound 
$$ |f_j(r)|^p_{L^2_t(H)} \leq \biggl(\sum_j |f_j(r)|^2_{L^2_t(H)} \biggr)^{p/2}, $$
will finish the proof making use of the disjointness of $\chi_{I_{m-j}}$ and the
summability of $2^{-|m|\delta(p)}$. The proof is similar for $\sum_m 
E^\alpha_{j,m-j}$ except that we may show the stronger $L^p(\mu_d)$ bounds.

For the case $\sum_m S^\alpha_{j,m}$, we shall use the fact that ${I_m^{**}}$ overlap only finitely many times to get
\begin{align*}
\pare{\sum_j \babs{ \sum_m \chi_{{I_m^{**}}} S^\alpha_{j,m} f_j }^2 }{2} 
&\les  \pare{\sum_m \chi_{{I_m^{**}}}  \sum_j \abs{ S^\alpha_{j,m} f_j }^2 }{2} \\
&\les  \sum_m \chi_{{I_m^{**}}}  \pare{ \sum_j \abs{ S^\alpha_{j,m} f_j }^2 }{2}.
\end{align*}

Thus, by Proposition \ref{prop:S},
\begin{align*}
&\bnorm{\biggl( \sum_j \babs{\sum_m S^\alpha_{j,m} f_j }^2_{H} \biggr)^{1/2} }_{L^{p}(\mu_d)} 
\les \pare{ \sum_m \bnorm{\pare{ \sum_j \abs{ S^\alpha_{j,m} f_j }_H^2 }{2}}_{L^p(\mu_d)}^p }{p} \\
&\les \pare{ \sum_m \bnorm{ \chi_{I_m}  \pare{ \sum_j \abs{f_j }_{L^2_t(H)}^2 }{2}}_{L^p(\mu_d)}^p }{p} 
=  \bnorm{\pare{ \sum_j \abs{f_j }_{L^2_t(H)}^2 }{2}}_{L^p(\mu_d)}.
\end{align*}
\end{proof}

\section{Kernel estimate}
\subsection{Estimate I}
The goal of this section is to obtain an estimate for the kernel of the operator $\mathcal{R}^\alpha_j $. Note that $\mathcal{R}^\alpha_j f(r)$ can be written as
\begin{align*}
&\int_0^\infty 2^{jd}\int_0^\infty \int_1^2  \biggl(1-\frac{\rho}{t}\biggr)_+^{\alpha-1} \chi\biggl(\frac{\rho}{t}\biggr) f_t(s) \frac{dt}{t} B_d(2^j r\rho) B_d(2^js\rho) d\mu_d(\rho)  d\mu_d(s) \\
&:= \int_0^\infty 2^{jd}K(2^jr,2^js)[f(s)] d\mu_d(s),
\end{align*}
where $K(r,s)$ is a bounded linear operator from $L^2_t(H):= L^2([1,2], dt/t, H)$ to $H$.

For $f\in L^2_t(H)$, define the operator $\mathcal{K}$ by
\begin{equation*}
\mathcal{K}[f](u) = \int_1^2 t\kappa(tu) f(t) \frac{dt}{t},
\end{equation*}
where $\fto \kappa (\rho) = (1-\rho)^{\alpha-1}_+ \chi (\rho)$. Then $K(r,s)[f]$ can be written as 
\begin{align*}
 K(r,s)[f] &= \int_0^\infty \int_1^2 \biggl(1-\frac{\rho}{t}\biggr)_+^{\alpha-1} \chi\biggl(\frac{\rho}{t}\biggr) f(t) \frac{dt}{t} B_d(r\rho) B_d(s\rho) d\mu_d(\rho)\\
 &= \int_0^\infty \fto[\mathcal{K}[f]](\rho) B_d(r\rho) B_d(s\rho) d\mu_d(\rho).
\end{align*}

We shall borrow the kernel estimate from \cite{GaSe} for the characterization of Hankel multipliers to obtain a rather precise estimate for the kernel $K(r,s)$. Here and in what follows, we set $\omega_N(u) = (1+|u|)^{-N}$. 
\begin{prop}\label{prop:kernel0}
Let $f \in L^2_t(H)$. Then for $N\in \N$,
\begin{equation} \label{eqn:kernel}
\abs{K(r,s)[f]}_H \les_N \sum_{\pm,\pm}\frac{W[f](\pm r \pm s)}{[ (1+r)(1+s)]^{(d-1)/2}},
\end{equation}
where $W[f](x) = |\mathcal{K}[f]|_H* \omega_N(x).$
\end{prop}
\begin{proof}
This follows from the kernel estimate in \cite[Proposition 3.1]{GaSe}; If $m= \fto k$ is integrable and compactly supported in $(0,\infty)$, then
\begin{equation}\label{eqn:kerGS}
\abs{\int_0^\infty \fto[k](\rho) B_d(\rho r)B_d(\rho s) d\mu_d(\rho)} \leq 
C_N\sum_{\pm,\pm}\frac{|k|*\omega_N(\pm r \pm s)}{[ (1+r)(1+s)]^{(d-1)/2}},
\end{equation}
where $C_N$ does not depend on $k,r,s$. 
\end{proof}

In the proof of \eqref{eqn:kerGS}, the following asymptotic formula is used.
\begin{equation} \label{eqn:asymptotic}
B_d(x) = \sum_{\pm} \sum_{\nu=0}^M c^\pm_{\nu,d} e^{\pm ix} x^{-\nu-\frac{d-1}{2}} + x^{-M} E_{M,d}(x)
\end{equation}
for $x\geq 1$ where $E_{M,d}$ have bounded derivatives. The terms $\pm r \pm s$ show up naturally by the translation arising from the oscillation term $e^{\pm ix}$. See \cite{GaSe} for details.

\subsection{Estimate II}
In what follows, we shall write $f^\alpha(t) = t^{-\alpha} f(t)$. The main result of this section is the following.
\begin{prop}\label{prop:Wf}
Let $f \in L^2_t(H)$. Then for all sufficiently large $N$,
\begin{equation}\label{eqn:Wf} 
W[f](x) \les \frac{1}{(1+|x|)^\alpha} |\ift[f^\alpha]|_H* \omega_N (x) + \frac{1}{(1+|x|)^{\alpha+1}} |f|_{L^2_t(H)}.
\end{equation}
In addition, there is a uniform estimate
\begin{equation}\label{eqn:Wfu}
W[f](x) \les |f|_{L^2_t(H)}.
\end{equation}
\end{prop}
\begin{proof} 
Since $W[f]$ is a convolution of $|\mathcal{K}[f]|_H$ and $\omega_N$, where $N$ may be chosen as large as we want, it suffices to prove 
\begin{align} \label{eqn:Vf} 
|\mathcal{K}[f](u)|_H &\les \frac{1}{(1+|u|)^\alpha} |\ift[f^\alpha]|_H* \omega_N (u) + \frac{1}{(1+|u|)^{\alpha+1}} |f|_{L^2_t(H)}, \\
\label{eqn:Vfu}
|\mathcal{K}[f](u)|_H &\les |f|_{L^2_t(H)}.
\end{align}
For instance, the inequality $(1+|x-u|)^{-1} \leq (1+|u|)(1+|x|)^{-1}$ may be used to obtain \eqref{eqn:Wf} from \eqref{eqn:Vf}. 

(\ref{eqn:Vfu}) follows from Cauchy-Schwarz inequality since $\kappa \in L^\infty(\R)$. 

For \eqref{eqn:Vf}, we need a standard asymptotic formula for $\kappa$. 
\begin{equation} \label{eqn:kappa}
\kappa(u) = \frac{\Gamma(\alpha)}{2\pi} e^{iu} (iu)^{-\alpha} + O(|u|^{-(\alpha+1)}) \text{ as }  u\to \pm \infty.
\end{equation}
We shall sketch a proof of (\ref{eqn:kappa}) at the end of this section. 

Let us assume $|u|\geq  C$ for some large constant $C$. By (\ref{eqn:kappa}),
\begin{align*}
|\mathcal{K}[f](u)|_H &\les \abs{\int_1^2 e^{itu}(itu)^{-\alpha} f(t) dt}_H + \abs{\int_1^2 O(|tu|^{-(\alpha+1)}) f(t) dt}_H \\
&\les |u|^{-\alpha} |\ift[f^\alpha] (u)|_H +  |u|^{-(\alpha+1)} |f|_{L^2_t(H)}.
\end{align*}
This estimate combined with (\ref{eqn:Wfu}) implies (\ref{eqn:Vf}).
\end{proof}

\begin{proof}[Proof of (\ref{eqn:kappa}), sketch] 
Consider the case $u \gg 1$.
By Fourier inversion formula and a change of variable, we may write
$$\kappa (u) = \frac{e^{iu}}{2\pi} \int_0^\infty \rho^{\alpha-1} \tilde \chi(\rho) e^{-iu\rho} d\rho, $$
where $\tilde{\chi}$ is a compactly supported smooth function which is $1$ near $0$. Introducing the factor $e^{-\rho}$, we split the integral as
\begin{align*}
I+II = \int_0^\infty \rho^{\alpha-1} e^{-(1+iu)\rho} d\rho + \int_0^\infty \rho^{\alpha-1}  (\tilde\chi(\rho)e^\rho-1) e^{-\rho} e^{-iu\rho} d\rho.
\end{align*}

For the main term, an application of Cauchy's theorem justifies the ``change of variable" $\rho \to (1+iu)^{-1}\rho$ to get 
$$ I= \frac{\Gamma(\alpha)}{(1+iu)^\alpha}. $$
For the error term, we have $|II| \les_N (1+|u|)^{-N}$ by integration by parts.

Finally, Taylor expansion gives the asymptotic formula. The case $u \to -\infty $ follows from the case $u \to \infty$ since $\kappa(u) = \overline{\kappa(-u)}$.
\end{proof}

\section{Proof of the main propositions} \label{sec:proof}
Throughout the section, we shall omit the summation notation $\sum_{\pm,\pm}$ in the kernel estimate \eqref{eqn:kernel}.
\subsection{Proof of Proposition \ref{prop:HE}}
By a scaling argument, we may assume that $j=0$. Indeed, we have 
$ V^\alpha_{j,m-j} f (r) = V^\alpha_{0,m} [f(2^{-j} \cdot)](2^jr).$

By \eqref{eqn:kernel},  
$$|V^\alpha_{0,m}f(r)|_H \les \int_{I_m} \frac{\chi_{R_{m+3}}(r)W[f(s)](\pm r \pm s)  }{(1+r)^{(d-1)/2}} \frac{d\mu_d(s)}{(1+s)^{(d-1)/2}}.$$
Then by Minkowski's inequality, $\norm{V^\alpha_{0,m}f}_{L^{p,2}(\mu_d,H)}$ is bounded by
\begin{equation}\label{eqn:eqnH1}
C \int_{I_m} \norm{\frac{\chi_{R_{m+3}}W[f(s)](\pm \cdot \pm s)  }{(1+\cdot)^{(d-1)/2}}}_{L^{p,2}(\mu_d)} \frac{d\mu_d(s)}{(1+s)^{(d-1)/2}}.
\end{equation}

Then the norm inside of the integral is bounded by
\begin{align}\label{eqn:inter}
\norm{\frac{\chi_{R_{m+2}}W[f(s)](\pm \cdot)  }{(1+\cdot)^{(d-1)/2}}}_{L^{p,2}(\mu_d)} \les
\norm{\frac{W[f(s)](\cdot)}{(1+|\cdot|)^{(d-1)/2}}}_{L^{p,2}(\nu_d)}
\end{align}
by a change of variable $r \to r \pm s$ and $(r\pm s) \sim r$, where $\nu_d(x) = (1+|x|)^{d-1}$ is a measure on $\R$.

We claim that \eqref{eqn:inter} is bounded by $C|f(s)|_{L^2_t(H)}$, which would imply
\begin{align*}
\norm{V^\alpha_{0,m}f}_{L^{p,2}(\mu_d,H)} &\les\int_{I_m} |f(s)|_{L^2_t(H)} \frac{d\mu_d(s)}{(1+s)^{(d-1)/2}} \\
&\les \norm{f \chi_{I_m}}_{L^{p,\infty}(\mu_d,L^2_t(H))} \norm{ \chi_{I_m} (1+\cdot)^{-(d-1)/2}}_{L^{p',1}(\mu_d)}
\end{align*}
by a variant of H\"older's inequality in Lorentz spaces (see \cite{Oneil}). The proof is complete if we observe that
$$\norm{ \chi_{I_m} (1+\cdot)^{-(d-1)/2}}_{L^{p',1}(\mu_d)} 
\leq \min( 2^{md/p'}, 2^{-m(d(1/p-1/2) -1/2)} ).$$
Here we have used the assumption that $1<p<2d/(d+1)$.

We turn to the proof of the claim. We separately estimate the main term and the error term given by Proposition \ref{prop:Wf}. For the error term, we control the $L^{p,2}$ norm by $L^p$ norm to obtain
\begin{align*}
\pare{ \int_{\R} (1+|x|)^{-p[\alpha+1 - (d-1)(1/p-1/2)]} dx}{p} |f(s)|_{L^2_t(H)} \les |f(s)|_{L^2_t(H)}
\end{align*}

For the main term, we apply H\"older's inequality and Plancherel's identity.
\begin{align*}
&\norm{\frac{|\ift[f^\alpha(s)]|_H*\omega_N (1+|\cdot|)^{-\frac{d-1}{2}} }{(1+|\cdot|)^{\alpha}}}_{L^{p,2}(\nu_d)} \\
&\les\norm{|\ift[f^\alpha(s)]|_H*\omega_N}_{L^2(\R)} \norm{ (1+|\cdot|)^{-\alpha}}_{L^{(\frac{1}{p}-\frac{1}{2})^{-1},\infty}(\nu_d)}\\
&\les \norm{\omega_N}_{L^1(\R)}\pare {\int_1^2 |f_t(s)|_H^2  t^{-2\alpha} dt}{2} \les|f(s)|_{L^2_t(H)}.
\end{align*}
For the second inequality, that $\alpha = d(\frac{1}{p}-\frac{1}{2})$ was required.

\subsection{Proof of Proposition \ref{prop:E}}
By scaling, we may assume that $j=0$.
\subsubsection{The case $m\leq 0$}
By \eqref{eqn:kernel} and the change of variable $s \to s \pm r$, 
\begin{align*}
|E^\alpha_{0,m}f(r)|_H 
& \les 2^{m(d-1)} \chi_{L_{m-2}}(r) \int_{I_m} W[f\chi_{I_m}(s)](\pm r \pm s)  ds \\
& \les 2^{m(d-1)} \chi_{L_{m-2}}(r) \int_{I_m^*} W[f\chi_{I_m}(s\pm r )](\pm s)  ds.
\end{align*}
By Minkowski's inequality,
$\norm{E^\alpha_{0,m}f}_{L^p(\mu_d,H)}$ is bounded by
\begin{align*}
&C2^{m(d-1)} 2^{m(d-1)/p} \int_{I_m^*}  \pare{ \int_{ L_{m-2}} W[f\chi_{I_m}(s\pm r)](\pm s)^p dr}{p}  ds \\
&\les 2^{m(d-1)} \int_{I_m^*}  \pare{ \int_{I_m^{**}} W[f\chi_{I_m}(r)](\pm s)^p r^{d-1} dr}{p}  ds,
\end{align*}
since $I_m^* \pm L_{m-2} \subset I_m^{**}.$
Applying the uniform estimate (\ref{eqn:Wfu}), we get
\begin{equation*}
\norm{E^\alpha_{0,m}f}_{L^p(\mu_d,H)} \les 2^{md} \norm{f\chi_{I_m}}_\lpt.
\end{equation*}

\subsubsection{The case $m\geq 0$}
By \eqref{eqn:kernel} and the change of variable $s \to s \pm r$, 
\begin{equation*}
|E^\alpha_{0,m}f(r)|_H  \les 2^{m(d-1)/2} \chi_{L_{m-2}(r)} (1+r)^{-(d-1)/2} \int_{I_m} W[f\chi_{I_m}(s\pm r)]( \pm s)  ds.
\end{equation*}
Then we take $L^p(\mu_d)$ norm, and then apply Minkowski's inequality. Using $r \leq 2^m$, $\norm{E^\alpha_{0,m}f}_{L^p(\mu_d,H)}$ is bounded by
\begin{equation} \label{eqn:norm2}
\begin{split}
&C2^{m(d-1)/p} \int_{I_m^*}  \pare{ \int_{ L_{m-2}} W[f\chi_{I_m}(s\pm r)](\pm s)^p dr}{p}  ds \\
&\les \int_{I_m^*}  \pare{ \int_{I_m^{**}} W[f\chi_{I_m}(r)](\pm s)^p r^{d-1} dr}{p}  ds.
\end{split}
\end{equation}

We estimate the main term and the error term separately, using (\ref{eqn:Wf}).  (\ref{eqn:norm2}) for the error term is bounded by a constant times
\begin{equation*}
\int_{I_m^*}  \pare{ \int_{I_m^{**}} |f\chi_{I_m}(r)|_{L^2_t(H)}^p r^{d-1} dr}{p}  \frac{ds}{(1+|s|)^{\alpha+1}} \leq 2^{-m\alpha}\norm{f\chi_{I_m}}_\lpt
\end{equation*}
which has the desired decay term $2^{-m\alpha}$.

The estimate for the main term requires the assumption $\alpha>1/2$. (\ref{eqn:norm2}) for the main term is bounded by a constant times
\begin{equation}\label{eqn:norm3}
\begin{split}
&\int_{I_m^*}  \pare{ \int_{I_m^{**}} [|\ift[[f\chi_{I_m}(r)]^\alpha]|_H* \omega_N (\pm s)]^p r^{d-1} dr}{p}  \frac{ds}{(1+|s|)^{\alpha}} \\
&\leq 2^{-m(\alpha-\frac{1}{2})}\biggl( \int_{I_m^*}  \biggl( \int_{I_m^{**}} [|\ift[[f\chi_{I_m}(r)]^\alpha]|_H* \omega_N (\pm s)]^p r^{d-1} dr \biggr)^{\frac{2}{p}} ds \biggr)^{1/2},
\end{split}
\end{equation}
where we applied Cauchy-Schwarz inequality for the $s$-integral.

Then by Minkowski's inequality, Young's inequality, and Plancherel's identity, (\ref{eqn:norm3}) is bounded by 
\begin{align*}
&2^{-m(\alpha-\frac{1}{2})} \biggl( \int_{0}^\infty \norm{\ift[[f\chi_{I_m}(r)]^\alpha]}_{L^2(\R,H)}^p r^{d-1} dr \biggr)^{1/p} \\
&\les2^{-m(\alpha-\frac{1}{2})} \norm{f\chi_{I_m}}_\lpt.
\end{align*}

\subsection{Proof of Proposition \ref{prop:S}}
We may assume that the function $\ve{f}$ is supported on $I_m$. To show that the estimate holds for $p=2$, it is enough to show that 
$$\norm{T^\alpha_{j,m}f_j}_{L^2( \mu_d,H)} \les \norm{f_j}_{L^2(\mu_d,L_t(H))}, $$
uniformly in $j$. But this easily follows from Plancherel's identity and Cauchy-Schwarz inequality since $\norm{\han[R^\alpha_{2^jt}]}_{L^\infty} \leq 1$.

Thus it suffices to prove a weak type inequality for $p=1$, namely that the inequality
\begin{equation} \label{eqn:weak1}
\mu_d \left( \{  r\in {I_m^{**}} : |\ve{S}^\alpha_{m}\ve{f}|_{l^2(H)} > \lambda  \} \right) \leq \frac{C}{\lambda} \bbnorm{\ve{f}}_{L^1(\mu_d,B)}
\end{equation}
holds for $\lambda>0$ and $B=l^2(L_t(H))$ by a vector valued version of the Marcinkiewicz interpolation theorem. 

We follow the usual strategy for proving weak type inequalities. We apply the Calder\'{o}n-Zygmund decomposition, Proposition \ref{prop:cz}, to get $\ve{f} = \ve{g} + \ve{b} = \ve{g} + \sum_\nu \ve{b_\nu}$, where $\ve{b}_\nu = \{ b_{\nu,j} \}_j$ is supported on $J_\nu \subset I_m$ and has cancellation. Let us denote by $J_\nu^*$ the interval with the same center as $J_\nu$ and twice its length and by $\Omega$ the union of $J_\nu^*$. 

Then (\ref{eqn:weak1}) for $\ve{g}$ can be shown as usual by the $L^2(\mu_d, B)$ boundedness of $\ve{S}^\alpha_m$. In addition, (\ref{eqn:weak1}) for $\ve{b}$ reduces to 
\begin{equation} \label{eqn:weak}
\mu_d \left( \{  r\in {I_m^{**}}\setminus \Omega : |\ve{S}^\alpha_{m}\ve{b}|_{l^2(H)} > \lambda  \} \right) \leq \frac{C}{\lambda} \bbnorm{\ve{f}}_{L^1(\mu_d,B)}.
\end{equation}
The left hand side of (\ref{eqn:weak}) is bounded by
\begin{equation} \label{eqn:weak3}
\lambda^{-1} \int_{{I_m^{**}}\setminus \Omega} |\ve{S}^\alpha_{m}\ve{b}(r)|_{l^1(H)} d\mu_d(r) \leq \lambda^{-1} \sum_\nu \sum_j \int_{{I_m^{**}}\setminus J_\nu^*} |S^\alpha_{j,m} b_{\nu,j}(r)|_{H} d\mu_d(r).
\end{equation}

Let us denote the integral on the right hand side of \eqref{eqn:weak3} by $\mathcal{I}_{j,\nu}$. We claim that there is $\epsilon>0$ such that
\begin{equation} \label{eqn:claimS}
\mathcal{I}_{j,\nu} \les \min (2^j|J_\nu|, [2^j |J_\nu|]^{-\epsilon}) \norm{b_{\nu,j}}_{L^1(\mu_d,L^2_t(H))}.
\end{equation}
Then (\ref{eqn:claimS}) implies
\begin{align*}
\sum_j \mathcal{I}_{j,\nu} \les \sum_j \min (2^j|J_\nu|, [2^j |J_\nu|]^{-\epsilon}) \bbnorm{\ve{b_\nu}}_{L^1(\mu_d,B)} \les\bbnorm{\ve{b_\nu}}_{L^1(\mu_d,B)}.
\end{align*}
Then (\ref{eqn:weak3}) is bounded by
\begin{equation*}
C \lambda^{-1} \sum_\nu \bbnorm{\ve{b_\nu}}_{L^1(\mu_d,B)}
\les \sum_\nu \mu_d (J_\nu) \les \lambda^{-1} \bbnorm{\ve{f}}_{L^1(\mu_d,B)},
\end{equation*}
as desired by Proposition \ref{prop:cz}.

\begin{proof}[Proof of the claim (\ref{eqn:claimS})]
By the kernel estimate \eqref{eqn:kernel},
\begin{align*}
|S^\alpha_{j,m} b_{\nu,j}(r)|_{H} &\les \chi_{{I_m^{**}}}(r) \int_{J_\nu} 
\frac{2^{jd} W[b_{\nu,j}(s)] (2^j(\pm r \pm s)) d\mu_d(s)}{[(1+2^jr)(1+2^js)]^{(d-1)/2}} \\
&\les \frac{\chi_{{I_m^{**}}}(r)}{2^{m(d-1)}}  \int_{J_\nu} 
2^j W[b_{\nu,j}(s)] (2^j(\pm r \pm s)) d\mu(s)
\end{align*}
using $r\sim s \sim 2^m$. Then 
\begin{equation}\label{eqn:mon}
\begin{split}
\mathcal{I}_{j,\nu} &\les \int_{J_\nu} \int_{{I_m^{**}}\setminus J_\nu^*} 2^j W[b_{\nu,j}(s)] (2^j(\pm r \pm s)) dr d\mu_d(s) \\
&\les \int_{J_\nu} \int_{|x|\geq |J_\nu|/2} 2^j W[b_{\nu,j}(s)] (2^jx) dx d\mu_d(s)\\
&\les \int_{J_\nu} \int_{|x|\geq 2^j|J_\nu|/2} W[b_{\nu,j}(s)] (x) dx d\mu_d(s).\\
&\les \int_{J_\nu} \int_{|x|\geq 2^j|J_\nu|/2} \frac{|\ift[b^\alpha_{\nu,j} (s)]|_H*\omega_N(x)}{(1+|x|)^\alpha} + \frac{|b_{\nu,j}(s)|_{L^2_t(H)}}{(1+|x|)^{\alpha+1} } dx d\mu_d(s).
\end{split}
\end{equation}
For the second inequality, we used the fact that $$|\pm r \pm s| \geq |r-s| \geq |J_\nu|/2,$$ whenever $r\in {I_m^{**}}\setminus J_\nu^*$ and $s\in J_\nu$. For the last inequality, we used Proposition \ref{prop:Wf}.

Choose $\epsilon = (\alpha-1/2)/2$. Then the main term of (\ref{eqn:mon}) is bounded by 
\begin{equation} \label{eqn:argue}
\begin{split}
&C (2^j |J_\nu|)^{-\epsilon} \int_{J_\nu} \int_{\R}  \frac{|\ift[b^\alpha_{\nu,j} (s)]|_H*\omega_N(x)}{(1+|x|)^{\alpha-\epsilon}} dxd\mu_d(s)   \\
\les& (2^j |J_\nu|)^{-\epsilon} \int_{J_\nu}  \norm{|\ift[b^\alpha_{\nu,j} (s)]|_H*\omega_N}_{L^2(\R)} d\mu_d(s) \\
\les& (2^j |J_\nu|)^{-\epsilon} \int_{J_\nu}  \norm{b_{\nu,j} (s)}_{L^2_t(H)} d\mu_d(s) =  (2^j |J_\nu|)^{-\epsilon}   \norm{b_{\nu,j}}_{L^1(\mu_d,L^2_t(H))}.
\end{split}
\end{equation}
where we have applied Cauchy-Schwarz, Young's inequality, and Plancherel's identity. The estimation of the error term of (\ref{eqn:mon}) is straightforward.

Next we seek an inequality which is good when $2^j|J_\nu|$ is small. Let $s_\nu$ be the center of the interval $J_\nu$. Using the cancellation of $b_{\nu,j}$, we may write
\begin{equation*}
S^\alpha_{j,m} b_{\nu,j}(r) = \chi_{{I_m^{**}}}(r) \int_{J_\nu} 2^{jd}
[K(2^jr,2^js)-K(2^jr,2^js_\nu)][b_{\nu,j}(s)] d\mu_d (s).
\end{equation*}

Let $f$ be a $L^2_t(H)$-valued function. Then $[K(2^jr,2^js)-K(2^jr,2^js_\nu)][f]$ can be written as
\begin{equation*}
\begin{split}
&\int \fto[\mathcal{K}[f]](\rho)  B_d(2^jr\rho) [B_d(2^js\rho)-B_d(2^js_\nu\rho)] d\mu_d(\rho) \\
=& 2^j(s-s_\nu) \int_0^1 \left[\int \fto[\mathcal{K}[f]](\rho)  B_d(2^jr\rho)  B_d'(2^js(\tau)) \rho^d d\rho \right] d\tau, \\
\end{split}
\end{equation*}
where $s(\tau) = s_\nu + \tau(s-s_\nu)$. 

We remark that $B_d'$ has the same asymptotic expansion as $B_d$ given in  \eqref{eqn:asymptotic}, thus the expression inside of the bracket follows the same estimate for $K(2^jr,2^j s(\tau))[f]$. Since $r\sim s(\tau) \sim 2^m$ for $\tau \in [0,1]$, an argument similar to the proof of the kernel estimate shows 
\begin{align*}
&2^{jd}|[K(2^jr,2^js)-K(2^jr,2^js_\nu)][f]|_H  \\
\les &2^j |J_\nu| \int_0^1 \frac{2^{jd}W[f](2^j (\pm r \pm s(\tau)))}{[(1+2^jr)(1+2^j s(\tau))]^{(d-1)/2}} d\tau \\
\les &\frac{2^j |J_\nu|}{2^{m(d-1)}} \int_0^1 2^jW[f](2^j (\pm r \pm s(\tau))) d\tau.
\end{align*}

Thus, 
\begin{equation}\label{eqn:mon2}
\begin{split}
\mathcal{I}_{j,\nu} &\les 2^j |J_\nu| \int_0^1 \int_{J_\nu} \int_{{I_m^{**}}\setminus J_\nu^*} 2^jW[b_{\nu,j}(s)](2^j (\pm r \pm s(\tau))) dr d\mu(s)  d\tau \\
&\les 2^j |J_\nu|\int_{J_\nu} \int_{\R} 2^j W[b_{\nu,j}(s)] (2^jx) dx d
\mu_d(s) \\
&\les 2^j |J_\nu| \norm{b_{\nu,j}}_{L^1(\mu_d,L^2_t(H))},
\end{split}
\end{equation}
where the last inequality follows by a change of variable and arguing as in (\ref {eqn:argue}).
\end{proof}

\section*{Appendix: A note on $g$-function}
\label{sec:appendix}
In this section, we shall give a proof of $L^p$ bounds of the square function $g_{\Phi}$ by Calder\'{o}n-Zygmund theory in the vector valued setting. We shall obtain an analogue of the gradient condition for the Hankel convolution operator $T^K f := K \con f$. The materials to be discussed are quite standard and well-known, but Lemma \ref{lem:gradient} does not seem to appear in the literature.

\subsection*{Calder\'{o}n-Zygmund decomposition}
Let $B$ be a Banach space, and $L^p(\mu_d,B)$ be the Bochner space, i.e.
$$\norm{f}_{L^p(\mu_d,B)}^p = \int_0^\infty |f(r)|_B^p d\mu_d(r),$$
for strongly measurable functions $f:\R_+ \to B$. Then there is a
Calder\'{o}n-Zygmund decomposition for the functions $f\in
L^1(\mu_d,B)$.
\begin{prop} \label{prop:cz} Let $f\in L^1(\mu_d,B)$ and $\lambda >0$. Then
there are dyadic intervals $J_{\nu}$ with disjoint interiors and a decomposition
$$f = g+b = g+\sum_{\nu} b_{\nu}$$ such that the following properties hold.
\ben
\item $|g(s)|_B \leq  C\lambda$ a.e. $s$ and $\norm{g}_{L^1(\mu_d,B)} \leq
\norm{f}_{L^1(\mu_d,B)}$.
\item $b_{\nu}$ is supported on $J_{\nu}$ and $\int b_{\nu}(s) d\mu_d(s) = 0$.
\item $\int |b_{\nu}(s)|_B d\mu_d(s) \les \lambda \mu_d(J_{\nu})$.
\item $\sum_{\nu} \mu_d(J_{\nu}) \les \lambda^{-1} \norm{f}_{L^1(\mu_d,B)}$.
\een
\end{prop}

A proof may be found in \cite{GaSe}, but we shall give a sketch of the proof. Split $f = \sum_j f_j$ where $f_j = f
\chi_{I_j}$ and $I_j = [2^j,2^{j+1})$. Define $F_j(r)$ by the equation
\begin{equation} \label{eqn:CZ}
2^{j(d-1)} F_j(r) = f_j(r) r^{d-1}
\end{equation}
and perform the usual Calder\'{o}n-Zygmund
decomposition for the $B$-valued function $F_j$ by $F_j = G_j + B_j$. Then we
obtain $f_j = g_j + b_j$, where $g_j$ and $b_j$ are given by equations similar to \eqref{eqn:CZ}, and then sum in $j$.

\subsection*{The square function $g_{\Phi}$}
Consider the Hankel convolution operator
\begin{equation*}
T^Kf(r) = K\con f(r)=\int_0^\infty \tau^s K(r) f(s) d\mu_d(s),
\end{equation*}
where $\tau^s$ is the generalized translation
given by $\han[\tau^s f](\rho) = B_d(s\rho) \han{f}(\rho)$. See \eqref{eqn:gent1} and \eqref{eqn:gent2} for an explicit formula for $\tau^s$.

As in the Euclidean case, see e.g. \cite{RRT}, one may extend this to the vector-valued setting. Let $A$ and $B$ Banach spaces, and $K:\R_+ \to \mathcal{L}(A,B)$ be an operator valued
kernel, where $\mathcal{L}(A,B)$ is the space of bounded linear operators from $A$ to
$B$. 

\begin{lem}[Calder\'{o}n-Zygmund] Suppose that $T^K$ is a bounded operator from
$L^r(\mu_d,A)$ to $L^r(\mu_d,B)$ for some $r$, $1\leq r\leq \infty$. In addition,
suppose that
\begin{equation} \label{eqn:horh}
\int_{|r-s| > 2|s-\bar{s}|} |\tau^s K(r) - \tau^{\bar{s}} K(r)|_{\mathcal{L}(A,B)}
d\mu_d(r) \leq C.
\end{equation}
Then $T^K$ is bounded from $L^p(\mu_d,A)$ to $L^p(\mu_d,B)$ for $1<p<\infty$,
and there is a weak-type inequality
\begin{equation*}
 \mu_d (\{ r\in \R_+: |T^K f(r)|_B > \lambda\} )\leq \frac{C}{\lambda}
\norm{f}_{L^1(\mu_d,A)}. 
\end{equation*}
\end{lem}

This can be shown by using Proposition \ref{prop:cz}, or by the general theory for spaces of homogeneous type. One may ask an analogue of the gradient condition for the H\"{o}rmander condition \eqref{eqn:horh}. 
\begin{lem}\label{lem:gradient} The following condition
\begin{equation}\label{eqn:gradh}
|K'(r)|_{\mathcal{L}(A,B)} \leq C r^{-(d+1)}
\end{equation}
implies (\ref{eqn:horh}).
\end{lem}

Before we turn to the proof of this fact, we give an application
for Littlewood-Paley square functions. Note that the $g$-function defined in Section \ref{sec:Hankel} can be regarded as a vector-valued convolution operator
\begin{equation*}
 g_\Phi f(r) = \han [\Phi (\cdot/t) ] \con f (r),
\end{equation*}
where we regard $\han [\Phi (\cdot/t) ]$ as an operator valued kernel taking values in $\mathcal{L}(H, \tilde{H})$ for $H$-valued functions $f$, where $\tilde{H} = L^2( \Rp, \frac{dt}{t},H)$.

\begin{thm} \label{thm:gfuncH} Let $\Phi \in \s(\Rp)$ with $\Phi(0)=0$, and $\tilde{H}$ and $g_\Phi$ be as described above. Then for $1<p<\infty$,
$$C_p^{-1} \norm{f}_{L^p(\mu_d, H)} \leq \norm{g_\Phi f}_{L^p(\mu_d,\tilde{H})} \leq
C_p \norm{f}_{L^p(\mu_d, H)}.$$
\end{thm}

\begin{proof}
We prove the second inequality. The first inequality follows from the
second inequality via the polarization identity. We may assume that $H=l^2$ and
$f=\{f_j\}_j$, since, for instance, we may write a $H$-valued function $f$ as the sum
$f(x)=\sum_j f_j(x) e_j$ for an orthonormal basis $\{e_j\}_j$, then using
Parseval's identity. 

First we consider the case $p=2$. By Plancherel's identity,
\begin{align*}
\norm{g(f)}_{L^2(\mu_d,\tilde{H})}^2 &=  \int_0^\infty \sum_j \norm{\han[\Phi(\cdot/t)]
\con f_j}_{L^2(\mu_d)}^2 \frac{dt}{t} \\
&= \biggl( \int_0^\infty |\Phi (t)|^2 \frac{dt}{t} \biggr) \sum_j
\norm{f_j}_{L^2(\mu_d)}^2 = C \norm{f}_{L^2(\mu_d,H)}^2.
\end{align*}

Next we verify \eqref{eqn:gradh}. Let $K=\han [\Phi(\cdot/t)]=t^d \han[\Phi](t\cdot).$ Since $\Phi \in S(R_+)$,
$$|K'(r)|_{\mathcal{L}(H,\tilde{H})} = |K'(r)|_{L^2(\Rp,\frac{dt}{t})} = \biggl( \int_0^\infty
|t^{d+1} \han[\Phi]'(tr)|^2 \frac{dt}{t} \biggr)^{1/2} = \frac{C}{r^{d+1}}.$$
\end{proof}

For the proof of Lemma \ref{lem:gradient}, we need explicit
formulae for the generalized translation. In what follows, we shall ignore
multiplicative constants, and write $A=B$ if $A=CB$ for a constant $C$ depends
only on $d$. One has
\begin{equation} \label{eqn:gent1}
\tau^s f(r) = \int_{0}^{\pi} f( (r,s)_{\theta} ) d\nu(\theta),
\end{equation}
where $(r,s)_{\theta} = (r^2 + s^2 - 2rs\cos \theta )^{1/2}$ and $d\nu(\theta)$
is a probability measure on $[0,\pi]$. One may also write 
\begin{equation} \label{eqn:gent2}
\tau^s f(r) = \int_{|r-s|}^{r+s} f(t) dW_{r,s}(t),
\end{equation}
where $dW_{r,s}(t)$ is a probability measure on $[|r-s|,r+s]$. See \cite{GoSt}.

\begin{proof}[Proof of Lemma \ref{lem:gradient}]
This observation is a combination of estimates from \cite{GoSt}, where an
analogue of the H\"{o}rmader-Mihlin multiplier theorem for Hankel multipliers is
proved. We shall denote by $|\cdot|$ the operator norm $|\cdot|_{\mathcal{L}(A,B)}$. By (\ref{eqn:gent1}),
\begin{align*}
&\int_{|r-s| \geq 2|s-\bar{s}|} |\tau^s K(r) - \tau^{\bar{s}} K(r)| d\mu_d(r)\\
&= \int_{|r-s| \geq 2|s-\bar{s}|} \abs{ \int_0^\pi K( (r,s)_\theta ) - K(
(r,\bar{s})_\theta ) d\nu(\theta) } d\mu_d(r) \\
&= \int_{|r-s| \geq 2|s-\bar{s}|} \abs{ \int_0^\pi \int_0^1 \frac{d}{dt} [K( (r,
ts+(1-t)\bar{s} )_\theta )] dt d\nu(\theta)  } d\mu_d(r).
\end{align*}

Let $\Phi (t) = (r, ts+(1-t)\bar{s} )_\theta$. Then $|\Phi'(t)| \leq |s-\bar{s}|$ (See \cite{GoSt}). Therefore, the last integral is bounded by
\begin{equation}\label{eqn:lep}
\begin{split}
&|s-\bar{s}| \int_0^1 \int_{|r-s| \geq 2|s-\bar{s}|} \int_0^\pi |K'|( (r,
ts+(1-t)\bar{s} )_\theta ) d\nu(\theta)   d\mu_d(r) dt \\
&\leq |s-\bar{s}| \int_0^1 \int_0^\infty \tau^{ts+(1-t)\bar{s}} |K'|(r) \chi(r)
d\mu_d(r) dt, 
\end{split}
\end{equation}
where $\chi(r)$ is the characteristic function of the set $\{r: |r-s| \geq
2|s-\bar{s}| \}$. 

Next, we use the identity 
$\int \tau^s f(r) g(r) d\mu_d(r) = \int f(r) \tau^s g(r) d\mu_d(r)$ and then
analyse $\tau^{ts+(1-t)\bar{s}} \chi(r)$. It follows by considering
(\ref{eqn:gent2}) that $$\tau^{ts+(1-t)\bar{s}} \chi(r) \leq
\chi_{[|s-\bar{s}|,\infty)}(r),$$ for any $t\in [0,1]$ as was observed in \cite{GoSt}. Thus,
\eqref{eqn:gradh} implies that \eqref{eqn:lep} is bounded by 
\begin{align*}
&|s-\bar{s}| \int_{|s-\bar{s}|}^\infty |K'(r)| d\mu_d(r) \\
&\les |s-\bar{s}| \int_{|s-\bar{s}|}^\infty r^{-(d+1)} r^{d-1} dr  \leq C.
\end{align*}
\end{proof}

\subsection*{Acknowledgements}
The author would like to thank his advisor Andreas Seeger for suggesting the problems, his guidance, and many helpful discussions. The author would like to thank B{\l}a{\.z}ej Wr{\'o}bel for bringing the multivariate Bochner-Riesz type multipliers in Section \ref{sec:Boch} to his attention.

\end{document}